\documentclass[final]{dmtcs-episciences}

\usepackage[utf8]{inputenc}
\usepackage[T1]{fontenc}
\usepackage{subfigure}

\usepackage{mathtools}
\usepackage{cleveref}

\newcommand{\N}{\mathbb{N}}

\newcommand{\cH}{\mathcal{H}}
\newcommand{\cR}{\mathcal{R}}
\newcommand{\cA}{\mathcal{A}}

\newcommand{\bp}{\mathbf{p}}
\newcommand{\br}{\mathbf{r}}
\newcommand{\bw}{\mathbf{w}}
\newcommand{\bh}{\mathbf{h}}
\newcommand{\bg}{\mathbf{g}}

\newcommand{\bx}{\mathbf{x}}

\graphicspath{{images/}}

\makeatletter 
\def\blfootnote{\gdef\@thefnmark{}\@footnotetext}
\makeatother

\newtheorem{theorem}{Theorem}
\newtheorem{lemma}{Lemma}
\newtheorem{proposition}{Proposition}
\newtheorem{observation}{Observation}
\newtheorem{corollary}{Corollary}

\author[Václav Blažej, Pavel Dvořák, Michal Opler]{V\'{a}clav Bla\v{z}ej\affiliationmark{1}\thanks{Acknowledges the support of the OP VVV MEYS funded project \mbox{CZ.02.1.01/0.0/0.0/16\_019/0000765} ``Research Center for Informatics''. This work was supported by the Grant Agency of the Czech Technical University in Prague, grant \mbox{No.~SGS20/208/OHK3/3T/18}.}
  \and Pavel Dvořák\affiliationmark{2}\thanks{Supported by Czech Science Foundation GA{\v C}R grant \#19-27871X.}
  \and Michal Opler\affiliationmark{2}\thanks{The work was supported by the grant SVV–2020–260578.}}

\title{Bears with Hats and Independence Polynomials}

\affiliation{
  Faculty of Information Technology, Czech Technical University in Prague, Prague, Czech Republic\\
  Faculty of Mathematics and Physics, Charles University, Prague, Czech Republic}


\begin{document}

\publicationdata{vol. 25:2 }{2023}{7}{10.46298/dmtcs.10802}{2023-01-12; 2023-01-12; 2023-05-30; 2023-07-04}{2023-07-04}
\maketitle
\begin{abstract}
    Consider the following hat guessing game.
    A bear sits on each vertex of a graph $G$, and a demon puts on each bear a hat colored by one of $h$ colors.
    Each bear sees only the hat colors of his neighbors.
    Based on this information only, each bear has to guess $g$ colors and he guesses correctly if his hat color is included in his guesses.
    The bears win if at least one bear guesses correctly for any hat arrangement.

    We introduce a new parameter---fractional hat chromatic number $\hat{\mu}$, arising from the hat guessing game.
    The parameter $\hat{\mu}$ is related to the hat chromatic number which has been studied before.
    We present a surprising connection between the hat guessing game and the independence polynomial of graphs.
    This connection allows us to compute the fractional hat chromatic number of chordal graphs in polynomial time, to bound fractional hat chromatic number by a function of maximum degree of $G$, and to compute the exact value of $\hat{\mu}$ of cliques, paths, and cycles.
\end{abstract}

\blfootnote{An extended abstract of this work has been published in the proceedings of WG 2021  as \cite{WGversion}.}

\section{Introduction}
In this paper, we study a variant of a hat guessing game.
In these types of games, there are some entities---players, pirates, sages, or, as in our case, bears.
A bear sits on each vertex of graph $G$.
There is some adversary (a demon in our case) that puts a colored hat on the head of each bear.
A bear on a vertex $v$ sees only the hats of bears on the neighboring vertices of $v$ but he does not know the color of his own hat.
Now to defeat the demon, the bears should guess correctly the color of their hats.
However, the bears can only discuss their strategy before they are given the hats.
After they get them, no communication is allowed, each bear can only guess his hat color.
The variants of the game differ in the bears' winning condition.

The first variant was introduced by Ebert~\cite{Ebert1998Thesis}.
In this version, each bear gets a red or blue hat (chosen uniformly and independently) and they can either guess a color or pass.
The bears see each other, i.e. they stay on vertices of a clique.
They win if at least one bear guesses his color correctly and no bear guesses a wrong color.
The question is what is the highest probability that the bears win achievable by some strategy.
Soon, the game became quite popular and it was even mentioned in NY Times~\cite{Robinsion2001NYTimes}.

Winkler~\cite{Winkler2002} studied a variant where the bears cannot pass and the objective is to maximize the number of bears that correctly guess their hat color.
A generalization of this variant for more than two colors was studied by Feige~\cite{Feige2004} and Aggarwal~\cite{Aggarwal2005Auctions}.
Butler et al.~\cite{Butler08Intro} studied a variant where the bears are sitting on vertices of a general graph, not only a clique.
For a survey of various hat guessing games, we refer to theses of Farnik~\cite{Farnik2015Thesis} or Krzywkowski~\cite{Krzywkowski2012Thesis}.

In this paper, we study a variant of the game introduced by Farnik~\cite{Farnik2015Thesis}, where each bear has to guess and they win if at least one bear guesses correctly.
He introduced a hat guessing number HG of a graph $G$ (also named as hat chromatic number and denoted $\mu$ in later works) which is defined as the maximum $h$ such that bears win the game with $h$ hat colors.
We study a variant where each bear can guess multiple times and we consider that a bear guesses correctly if the color of his hat is included in his guesses.
We introduce a parameter \emph{fractional hat chromatic number} $\hat{\mu}$ of a graph $G$, which we define as the supremum of $\frac{h}{g}$ such that each bear has $g$ guesses and they win the game with $h$ hat colors.

Although the hat guessing game looks like a recreational puzzle, connections to more ``serious'' areas of mathematics and computer science were shown---like coding theory~\cite{Ebert2003Coding,Jin19Codes}, network coding ~\cite{Gadouleau11NetworkCoding,Riis07NetworkCoding}, auctions~\cite{Aggarwal2005Auctions}, finite dynamical systems~\cite{Gadouleau18FDS}, and circuits~\cite{Wu09Circuits}.
In this paper, we exhibit a connection between the hat guessing game and the independence polynomial of graphs, which is our main result.
This connection allows us to compute the optimal strategy of bears (and thus the value of $\hat{\mu}$) of an arbitrary chordal graph in polynomial time.
We also prove that the fractional hat chromatic number $\hat{\mu}$ is equal, up to a logarithmic factor, to the maximum degree of a graph, i.e., $\hat{\mu}(G) = \Omega(\Delta / \log \Delta)$ and $\hat{\mu}(G) = O(\Delta)$.
Finally, we compute the exact value of $\hat{\mu}$ of graphs from some classes, like paths, cycles, and cliques.

We would like to point out that the existence of the algorithm computing $\hat{\mu}$ of a chordal graph is far from obvious.
Butler et al.~\cite{Butler08Intro} asked how hard is to compute $\mu(G)$ and the optimal strategy for the bears.
Note that a trivial non-deterministic algorithm for computing the optimal strategy (or just the value of $\mu(G)$ or $\hat{\mu}(G)$) needs exponential time because a strategy of a bear on $v$ is a function of hat colors of bears on neighbors of $v$ (we formally define the strategy in Section~\ref{sec:Prelim}).
It is not clear if the existence of a strategy for bears would imply a strategy for bears where each bear computes his guesses by some efficiently computable function (like linear, computable by a polynomial circuit, etc.).
This would allow us to put the problem of computing $\mu$ into some level of the polynomial hierarchy, as noted by Butler et al.~\cite{Butler08Intro}.
On the other hand, we are not aware of any hardness results for the hat guessing games.
The maximum degree bound for $\hat{\mu}$ does not imply an exact efficient algorithm computing $\hat{\mu}(G)$ as well.
This phenomenon can be illustrated by the edge chromatic number $\chi'$ of graphs.
By Vizing's theorem~\cite[Chapter 5]{DiestelBook}, it holds for any graph $G$ that $\Delta(G) \leq \chi'(G) \leq \Delta(G) + 1$.
However, it is NP-hard to distinguish between these two cases~\cite{Holyer81EdgeColoring}.

\noindent {\bf Organization of the Paper.}
We finish this section with a summary of results about the variant of the hat guessing game we are studying.
In the next section, we present notions used in this paper and we define formally the hat guessing game.
In Section~\ref{sec:Multiguess}, we formally define the fractional hat chromatic number $\hat{\mu}$ and compare it to $\mu$.
In Section~\ref{sec:FirstSteps}, we generalize some previous results to the multi-guess setting.
We use these tools to prove our main result in Section~\ref{sec:Polynomials} including the poly-time algorithm that computes $\hat{\mu}$ for chordal graphs.
The maximum degree bound for $\hat{\mu}$ and computation of exact values of paths and cycles are provided in Section~\ref{sec:Applications}.

\subsection{Related and Follow-up Works}
As mentioned above, Farnik~\cite{Farnik2015Thesis} introduced a hat chromatic number $\mu(G)$ of a graph $G$ as the maximum number of colors $h$ such that the bears win the hat guessing game with $h$ colors and played on $G$.
He proved that $\mu(G) \leq O\bigl(\Delta(G)\bigr)$ where $\Delta(G)$ is the maximum degree of $G$.

Since then, the parameter $\mu(G)$ was extensively studied.
The parameter $\mu$ for multipartite graphs was studied by Gadouleau and Georgiu~\cite{Gadouleau15Hats} and by Alon et al.~\cite{Alon2020HatGuesNum}.
Szczechla~\cite{Szczechla2017CycleGraphs} proved that $\mu$ of cycles is equal to $3$ if and only if the length of the cycle is $4$ or it is divisible by $3$ (otherwise it is $2$).
Bosek et al.~\cite{Bosek2019Bears} gave bounds of $\mu$ for some graphs, like trees and cliques.
They also provided some connections between $\mu(G)$ and other parameters like chromatic number and degeneracy.
They conjectured that $\mu(G)$ is bounded by some function of the degeneracy $d(G)$ of the graph $G$.
They showed that such function has to be at least exponential as for every $d \geq 1$ they presented a graph $G$ of $d(G) = d$ such that $\mu(G) \geq 2^{d}$.
This result was improved by He and Li~\cite{He2020Degenerate} who showed that for every $d \geq 1$ there is a graph $G$ of $d(G) = d$ and $\mu(G) \geq 2^{2^{d(G) - 1}}$.
Since $\hat{\mu}(G)$ is lower-bounded by $\Omega\bigl(\Delta(G)/ \log \Delta(G)\bigr)$ (as we show in Section~\ref{sec:Applications}) it holds that $\hat{\mu}$ can not be bounded by any function of degeneracy as there are graph classes of unbounded maximum degree and bounded degeneracy (e.g. trees or planar graphs).
Recently, Kokhas et al.~\cite{Kokhas2021CliquesI,Kokhas2021CliquesII} studied a non-uniform version of the game, i.e., every bear may have a different number of possible hat colors.
They considered cliques and almost cliques.
They also provided a technique to build a strategy for a graph $G$ whenever $G$ is made up by combining $G_1$ and $G_2$ with known strategies.
We generalize some of their results and use them as ``basic blocks'' for our main result.

After the presentation of the preliminary version of this paper~\cite{WGversion}, Latyshev and Kokhas~\cite{Latyshev2021Followup} extended ideas presented in this paper to reason about the standard hat chromatic number.
In particular, they found a family of graphs of unbounded maximum degree such that for each graph $G$ in the family holds that $\mu(G) = \frac{4}{3}\Delta(G)$; thus they disproved a conjecture that $\mu(G) \leq \Delta(G) + 1$ stated in Bosek, et al.~\cite{Bosek2019Bears} and Farnik~\cite{Farnik2015Thesis} that was previously noticed by Alon et al.~\cite{Alon2020HatGuesNum}.

\section{Preliminaries}
\label{sec:Prelim}
We use standard notions of the graph theory.
For an introduction to this topic, we refer to the book by Diestel~\cite{DiestelBook}.
We denote a clique as $K_n$, a cycle as $C_n$, and a path as $P_n$, each on $n$ vertices.
The maximum degree of a graph $G$ is denoted by $\Delta(G)$, where we shorten it to $\Delta$ if the graph $G$ is clear from the context.
The neighbors of a vertex $v$ are denoted by $N(v)$.
We use $N[v]$ to denote the closed neighborhood of $v$, i.e. $N[v] = N(v) \cup \{v\}$.
For a set $U$ of vertices of a graph $G$, we denote by $G \setminus U$ a graph induced by vertices $V(G) \setminus U$, i.e., a graph arising from $G$ by removing the vertices in $U$.

A \emph{hat guessing game} is a triple $\cH = (G, h, g)$ where
\begin{itemize}
\item $G=(V,E)$ is an undirected graph, called the \emph{visibility graph},
\item $h \in \N$ is a \emph{hatness} that determines the number of different possible hat colors for each bear, and
\item $g \in \N$ is a \emph{guessing number} that determines the number of guesses each bear is allowed to make.
\end{itemize}

The rules of the game are defined as follows.
On each vertex of $G$ sits a bear.
The demon puts a hat on the head of each bear.
Each hat has one of $h$ colors.
We would like to point out, that it is allowed that bears on adjacent vertices get a hat of the same color.
The only information the bear on a vertex $v$ knows are the colors of hats put on bears sitting on neighbors of $v$.
Based on this information only, the bear has to guess a set of $g$ distinct colors according to a deterministic strategy agreed to in advance.
We say bear \emph{guesses correctly} if he included the color of his hat in his guesses.
The bears win if at least one bear guesses correctly.

Formally, we associate the colors with natural numbers and say that each bear can receive a hat colored by a color from the set $S = [h] = \{0, \ldots, h-1\}$.
A \emph{hats arrangement} is a function $\varphi\colon V \to S$.
A strategy of a bear on $v$ is a function $\Gamma_v\colon {S}^{|N(v)|} \to \binom{S}{g}$, and a \emph{strategy for $\cH$} is a collection of strategies for all vertices, i.e. $(\Gamma_v)_{v\in V}$.
We say that a strategy is \emph{winning} if for any possible hats arrangement $\varphi\colon V \to S$ there exists at least one vertex $v$ such that $\varphi(v)$ is contained in the image of $\Gamma_v$ on $\varphi$, i.e., $\varphi(v) \in \Gamma_v \bigl( (\varphi(u))_{u \in N(v)} \bigr)$.
Finally, the game $\cH$ is \emph{winning} if there exists a winning strategy of the bears.

As a classical example, we describe a winning strategy for the hat guessing game $(K_3, 3, 1)$.
Let us denote the vertices of $K_3$ by $v_0$, $v_1$ and $v_2$ and fix a hats arrangement $\varphi$.
For every $i \in [3]$, the bear on the vertex $v_i$ assumes that the sum $\sum_{j \in [3]} \varphi(v_j)$ is equal to $i$ modulo $3$ and computes its guess accordingly.
It follows that for any hat arrangement $\varphi$ there is always exactly one bear that guesses correctly, namely the bear on the vertex $v_i$ for $i = \sum_j \varphi(v_j) \pmod{3}$.

Some of our results are stated for a non-uniform variant of the hat guessing game.
A non-uniform game is a triple $\bigl(G = (V,E),  \bh, \bg\bigr)$ where $\bh = (h_v)_{v \in V}$ and $\bg = (g_v)_{v \in V}$ are vectors of natural numbers indexed by the vertices of $G$ and a bear on $v$ gets a hat of one of $h_v$ colors and is allowed to guess exactly $g_v$ colors.
Other rules are the same as in the standard hat guessing game.
To distinguish between the uniform and non-uniform games, we always use plain letters $h$ and $g$ for the hatness and the guessing number, respectively, and bold letters (e.g. $\bh,\bg$) for vectors indexed by the vertices of $G$.

For our proofs we use two classical results.
First one is the inclusion-exclusion principle for computing a size of a union of sets.

\begin{proposition}[folklore]
 \label{prp:IncExc}
 For a union $A$ of sets $A_1,\dots,A_n$, it holds that
 \[
  |A| = \sum_{\emptyset \neq I \subseteq \{1,\dots,n\}}  (-1)^{|I|+1} \left| \bigcap_{i \in I} A_i \right|.
 \]
\end{proposition}

The other one is the rational root theorem, which we use to derive an algorithm for computing an exact value of $\hat{\mu}$, if the value is rational.

\begin{theorem}[Rational root theorem~\cite{Larson07Calculus}]
\label{thm:RationalRoot}
If a polynomial $a_n x^n + \dots a_1 x + a_0$ has integer coefficients, then every rational root is of the form $p/q$
where $p$ and $q$ are coprimes, $p$ is a divisor of $a_0$, and $q$ is a divisor of $a_n$.
\end{theorem}

\section{Fractional Hat Chromatic Number}
\label{sec:Multiguess}
From the hat guessing games, we can derive parameters of the underlying visibility graph $G$.
Namely, the \emph{hat chromatic number $\mu(G)$} is the maximum integer $h$ for which the hat guessing game $(G, h, 1)$ is winning, i.e., each bear gets a hat colored by one of $h$ colors and each bear has only one guess---we call such game a single-guessing game.
In this paper, we study a parameter \emph{fractional hat chromatic number $\hat{\mu}(G)$} which arises from the hat multi-guessing game and is defined as
\[
    \hat{\mu}(G) = \sup \left\{\frac{h}{g} \;\middle|\; (G, h, g) \text{ is a winning game }\right\}.
\]

Observe that $\mu(G) \le \hat{\mu}(G)$.
Farnik~\cite{Farnik2015Thesis} and Bosek et al.~\cite{Bosek2019Bears} also study multi-guessing games.
They considered a parameter $\mu_g(G)$ that is the maximum number of colors $h$ such that the bears win the game $(G,h,g)$.
The difference between $\mu_g$ and $\hat{\mu}$ is the following.
If $\mu_g(G) \geq k$, then the bears win the game $(G,k,g)$ and $\hat{\mu} \ge \frac{k}{g}$.
If $\hat{\mu}(G) \geq \frac{p}{q}$, then there are $h,g \in \N$ such that $\frac{p}{q} = \frac{h}{g}$ and the bears win the game $(G,h,g)$.
However, it does not imply that the bears would win the game $(G,p,q)$.
In this section, we prove that if the bears win the game $(G,h,g)$ then they win the game $(G,kh,kg)$ for any constant $k \in \N$.
The opposite implication does not hold---we discuss a counterexample at the end of this section.
Unfortunately, this property prevents us from using our algorithm, which computes $\hat{\mu}$, to compute also $\mu$ of chordal graphs.

Moreover, by definition, the parameter $\hat{\mu}$ does not even have to be a rational number.
In such a case, for each $p,q \in \N$, it holds that
\begin{itemize}
 \item If $\frac{p}{q} < \hat{\mu}(G)$ then there are $h,g \in \N$ such that $\frac{p}{q} = \frac{h}{g}$ and the bears win the game $(G,h,g)$.
 \item If $\frac{p}{q} > \hat{\mu}(G)$ then the demon wins the game $(G,p,q)$.
\end{itemize}
For example, the fractional hat chromatic number $\hat{\mu}(P_3)$ of the path $P_3$ is irrational.
In the case of an irrational $\hat{\mu}(G)$, our algorithm computing the value of $\hat{\mu}$ of chordal graphs outputs an estimate of $\hat{\mu}(G)$ with arbitrary precision.
We finish this section with a proof that the multi-guessing game is in some sense monotone.
\begin{observation}
 \label{obs:Monotonicity}
 Let $k \in \N$.
 If a game $\cH = (G, h, g)$ is winning, then the game $\cH_k = (G, k\cdot h, k \cdot g)$ is winning as well.
\end{observation}
\begin{proof}
 We derive a winning strategy for the game $\cH_k$ from a winning strategy for $\cH$.
 Each bear interprets a color in $[k\cdot h]$ as a pair $(i,c)$ where $i \in [k]$ and $c \in [h]$.
 Let $A_v$ be guesses of the bear on $v$ in the game $\cH$.
 For the game $\cH_k$, a strategy of the bear on $v$ is to make guesses $\bigl\{(i,c) \mid i \in [k], c \in A_v\bigr\}$.
 It is straight-forward to verify that this is a winning strategy for $\cH_k$.
\end{proof}

\begin{lemma}
\label{lem:Monotonicity}
 Let $\bigl(G = (V,E) ,h,g \bigr)$ be a winning hat guessing game.
 Let $r'$ be a rational number such that $r' \leq h/g$.
 Then, there exist numbers $h', g' \in \N$ such that $h'/g' = r'$ and the hat guessing game $(G,h',g')$ is winning.
\end{lemma}

\begin{proof}
 Let $p, q \in \N$ such that $r' = p/q$ and $\text{GCD}(p, q) = 1$.
 Let\footnote{GCD stands for the greatest common divisor and LCM stands for the least common multiple.} $\ell = \text{LCM}(h, p)$.
 
 Let $\bar{h} = \ell, \bar{g} = \ell \cdot g/h$.
 By Observation~\ref{obs:Monotonicity} for $k = \ell / h$, the game $(G,\bar{h},\bar{g})$ is winning.
 Let $h' = \ell$ and $g' = \ell\cdot q/p$.
 Since $p/q \leq h/g$ by the assumption, it holds that $g' \geq \bar{g}$.
 Thus, the bears have a strategy for $(G,h',g')$, as we increased the number of guesses and the hatness does not change ($h' = \bar{h} = \ell$).
 Moreover, $h'/g' = p/q = r'$. 
\end{proof}

It is straight-forward to prove a generalization of Lemma~\ref{lem:Monotonicity} for non-uniform games.
However, for simplicity, we state it only for the uniform games.
By the proof of the previous lemma, we know that we can use a strategy for $(G,h,g)$ to create a strategy for a game $(G, k\cdot h, k \cdot g + \ell)$ for arbitrary $k,\ell \in \N$ where $k \cdot g + \ell \le k\cdot h$.
However, it is unclear whether this also holds in general, i.e., given a winning strategy for a fractional hat chromatic number $h/g$, is it always possible to have a winning strategy for a decreased fraction $h'/g' < h/g$ where the hatness $h'$ and the guessing number $g'$ can be changed arbitrarily?
It is true for cliques.
We show in Section~\ref{sec:FirstSteps} that the bears win the game $(K_n, h, g)$ if and only if $h/g \leq n$.
However, it is not true in general.
For example, for $n$ large enough it holds that $\hat{\mu}(P_n) \geq 3$, as we show in Section~\ref{sec:Applications} that $\hat{\mu}(P_n)$ converges to 4 when $n$ goes to infinity.
However, Butler et al.~\cite{Butler08Intro} proved that $\mu(T) = 2$ for any tree $T$.
Thus, the bears lose the game $(P_n, 3, 1)$.

\section{Basic Blocks}
\label{sec:FirstSteps}
In this section, we generalize some results of Kokhas et al.~\cite{Kokhas2021CliquesI,Kokhas2021CliquesII} about cliques and strategies for graph products, which we use for proving our main result.
The single-guessing version of the next theorem (without the algorithmic consequences) was proved by Kokhas et al.~\cite{Kokhas2021CliquesI,Kokhas2021CliquesII}.
\begin{theorem}
\label{thm:Clique}
 Bears win a game $\bigl(K_n = (V,E), \bh, \bg\bigr)$ if and only if
 \[
  \sum_{v \in V} \frac{g_v}{h_v} \geq 1.
 \]
 Moreover, if there is a winning strategy, then there is a winning strategy $(\Gamma_v)_{v \in V}$ such that each $\Gamma_v$ can be described by two linear inequalities whose coefficients can be computed in linear time.
\end{theorem}
\begin{proof}
 The proof follows the proof of Kokhas et al.~\cite{Kokhas2021CliquesII} for the single-guessing game.
 First, suppose that $\sum_{v \in V} g_v/h_v < 1$ and fix some strategy of bears.
 A bear on $v$ guesses correctly the color of his hat in exactly $(g_v/h_v)$-fraction of all possible hat arrangements.
 Thus, if the sum is smaller than one, there is a hat arrangement where no bear guesses the color of his hat correctly.
 
 Now suppose the opposite inequality holds, i.e., $\sum_{v \in V} g_v/h_v \geq 1$.
 Let $V(K_n) = \{v_1,\dots,v_n\}$.
 For simplicity, we denote $h_i = h_{v_i}$ and $g_i = g_{v_i}$.
 Let $\ell = \text{LCM}(h_1,\dots,h_n)$ and $d_i = \ell/h_i$ (note that $d_i \in \N$).
 Let the bear on $v_i$ receive a hat of color $c_i \in [h_i]$, and let
 \[
  s = \sum_{1 \leq i \leq n} c_i\cdot d_i \pmod{\ell}.
 \]
 The bears cover the set $[\ell]$ by disjoint intervals $Q_i$ of length $d_i \cdot g_i$.
 A bear on $v_i$ makes his guesses according to a hypothesis that $s$ is in an interval $Q_i$ and we will show that he guesses correctly if $s \in Q_i$.
 More formally, for $b_i = \sum_{j < i} d_j\cdot g_j$ we define the interval $Q_i$ as $\{b_i,\dots, b_i + d_i\cdot g_i - 1\}$.
 Note that the union of intervals $Q_1,\dots,Q_{i-1}$ is exactly the set $[b_i]$.
 A bear on $v_i$ computes $s_i = \sum_{v \neq v_i} c_v \cdot d_v$.
 Then, he guesses all such colors $a_i$ such that $s_i + a_i\cdot d_i \pmod{\ell}$ is in $Q_i$.
 Since $Q_i$ contains $d_i \cdot g_i$ consecutive natural numbers and $\ell$ is divisible by $d_i$, he makes at most $g_i$ guesses.
 If $s$ is in $Q_i$ then the bear on $v_i$ guesses the color of his hat correctly, because $s = s_i + c_i\cdot d_i \pmod{\ell}$ and thus the bear on $v_i$ includes the color $c_i$ in his guesses.
 
 Note that the union $Q$ of all intervals $Q_i$ is exactly the set
 \[
   \left\{0,\dots, \sum_{1 \leq i \leq n} \frac{\ell\cdot g_i}{h_i} - 1\right\}.
 \]
 By assumption, we have that $\{0,\dots,\ell - 1\} \subseteq Q$.
 Since $0 \leq s < \ell$ by definition, it follows that $s$ has to be in some interval $Q_i$.
 
 For the ``moreover'' part, the bear on a vertex $v_i$ guesses all colors $a_i \in [h_i]$ such that
 \[
  b_i \leq (s_i + a_i \cdot d_i) \bmod \ell < b_i + d_i \cdot g_i.
 \]
 Observe that $s_i$ is a linear function of hat colors of bears sitting on the vertices different from $v$ and the coefficients $b_i$ and $d_j$ can be computed in linear time.
 \end{proof}

 By Theorem~\ref{thm:Clique}, we can conclude the following corollary.
 \begin{corollary}
 \label{cor:Clique}
  For each $n \in \N$, it holds that $\hat{\mu}(K_n) = n$.
 \end{corollary}

 Kokhas et al.~\cite{Kokhas2021CliquesI} provided another proof of analogue of Theorem~\ref{thm:Clique} for the single-guessing game, which can be generalized with similar ideas.
 However, the second proof does not imply a polynomial time algorithm for computing the strategy on cliques.
 For the interested reader, we provide the second proof of Theorem~\ref{thm:Clique} in Appendix~\ref{sec:Clique2}.

Further, we generalize a result of Kokhas and Latyshev~\cite{Kokhas2021CliquesI}.
In particular, we provide a new way to combine two hat guessing games on graphs $G_1$ and $G_2$ into a hat guessing game on graph obtained by gluing $G_1$ and $G_2$ together in a specific way.

Let $G_1 = (V_1, E_1)$ and $G_2 = (V_2, E_2)$ be graphs, let $S \subseteq V_1$ be a set of vertices inducing a clique in $G_1$, and let $v \in V_2$ be an arbitrary vertex of $G_2$.
The \emph{clique join of graphs $G_1$ and $G_2$ with respect to $S$ and $v$} is the graph $G = (V,E)$ such that $V = V_1 \cup V_2 \setminus \{v\}$; and $E$ contains all the edges of $E_1$, all the edges of $E_2$ that do not contain $v$, and an edge between every $w \in S$ and every neighbor of $v$ in $G_2$.
See Figure~\ref{fig:CliqueJoin} for a sketch of a clique join.

\begin{figure}[h]
  \centering
  \includegraphics{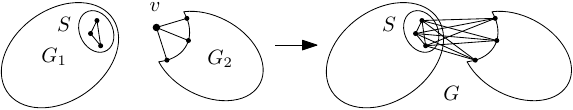}
  \caption{The clique join of graphs $G_1$ and $G_2$ with respect to $S$ and $v$.}%
  \label{fig:CliqueJoin}
\end{figure}

\begin{lemma}
\label{lem:CliqueJoin}
Let $\cH' = \bigl(G' =(V', E'), \bh', \bg'\bigr)$ and $\cH'' = \bigl(G'' = (V'', E''), \bh'', \bg''\bigr)$ be two hat guessing games and let $S \subseteq V'$ be a set inducing a clique in $G'$ and $v\in V''$.
Set $G$ to be the clique join of graphs $G'$ and $G''$ with respect to $S$ and $v$.
If the bears win the games $\cH'$ and $\cH''$, then they also win the game $\cH = (G, \bh, \bg)$ where
\begin{align*}
h_u =
\begin{cases}
  h'_u              &u\in V'\setminus S\\
  h''_u              &u\in V''\setminus \{v\}\\
  h'_u\cdot h''_v  &u \in S \text{, and}
\end{cases}\qquad
g_u =
\begin{cases}
  g'_u              &u\in V'\setminus S\\
  g''_u              &u\in V''\setminus \{v\}\\
  g'_u\cdot g''_v  &u \in S.
\end{cases}
\end{align*}
\end{lemma}

\begin{proof}
Using winning strategies $(\Gamma'_v)_{v \in V'}$ and $(\Gamma''_v)_{v \in V''}$ for $\cH'$ and $\cH''$ respectively, let us construct a winning strategy for $\cH$.
For every bear $u \in S$, we interpret his color as a tuple $(c'_u, c''_u)$ where $c'_u \in [h'_u]$ and $c''_u \in [h''_v]$.
Also, we define an imaginary hat color of the bear on vertex $v$ as $s = (\sum_{u \in S} c''_u) \bmod h''_v$.

Every bear on $ w \in V' \setminus S$ plays according to the strategy $\Gamma'_w$ using only the color $c'_u$ for his every neighbor $u \in S$.
Every bear on $w \in V'' \setminus \{v\}$ plays according to the strategy $\Gamma''_w$ using the imaginary hat color $s$ of $v$.
And finally, every bear on vertex $w \in S$ computes a set of guesses $A_w$ by playing the strategy $\Gamma'_w$ and a set of guesses $B$ by playing the strategy $\Gamma''_v$.
Since the bear on $w$ can see every other vertex of $S$, he computes the set
\[B_w = \left\{\left(c - \textstyle \sum_{u \in S \setminus \{w\}} c''_u \right) \bmod h''_v \mid c \in B \right\}.\]
Finally, the bear on $w$ guesses the set $A_w \times B_w$.

Fix an arbitrary hat arrangement.
In the simulated hat guessing game $\cH'$, there is a vertex $u_1$ such that the bear on $u_1$ guessed correctly.
If $u_1 \not\in S$ then it also guessed correctly in $\cH$.
Likewise, there is a bear on a vertex $u_2$ in the simulated hat guessing game $\cH''$ that guessed correctly and we are done if $u_2 \neq v$.
The remaining case is when $u_1 \in S$ and $u_2 = v$.
Thus, the bear on $v$ includes the color $s$ in his guesses in the game $\cH''$.
It follows that for each $w \in S$ holds that if $(c'_w,c''_w)$ is a hat color of the bear on $w$, then $c''_w \in B_w$.
Since $u_1 \in S$, the bear on $u_1$ includes his hat color $(c'_{u_1}, c''_v)$ in his guesses $A_{u_1} \times B_{u_1}$.
\end{proof}

We remark that Lemma~\ref{lem:CliqueJoin} generalizes Theorem 3.1 and Theorem 3.5 of \cite{Kokhas2021CliquesI} not only by introducing multiple guesses but also by allowing for more general ways to glue two graphs together.
Thus, it provides new constructions of winning games even for single-guessing games.

\begin{figure}[h]
  \centering
  \includegraphics{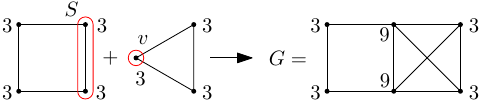}
  \caption{Applying Lemma~\ref{lem:CliqueJoin} on winning hat guessing games $(C_4, 3, 1)$ (see \cite{Szczechla2017CycleGraphs}) and $(K_3, 3, 1)$, we obtain a winning hat guessing game $(G, \bh, 1)$ where $G$ is the result of identifying an edge in $C_4$ and $K_4$, and $\bh$ is given in the picture.}
  \label{fig:CliqueJoinExample}
\end{figure}

\section{Independence Polynomial}
\label{sec:Polynomials}
The multivariate \emph{independence polynomial} of a graph $G = (V,E)$ on variables $\mathbf{x} = (x_v)_{v\in V}$ is
\[P_G(\bx) = \sum_{\substack{I \subseteq V\\\mathclap{I \text{ independent set}}}} \  \prod_{v \in I} x_v.\]

First, we describe the connection between the multi-guessing game and the independence polynomial informally and later prove the mentioned statements formally.
Consider the game $(G,h,g)$ and fix a strategy of bears.
Suppose that the demon put on the head of each bear a hat of random color (chosen uniformly and independently).
Let $A_v$ be an event that the bear on the vertex $v$ guesses correctly.
Then, the probability of $A_v$ is exactly $g/h$.
Moreover, for any independent set $I$ it holds that $A_v$ is independent on all events $A_w$ for $w \in I, w \neq v$.
Thus, we can use the inclusion-exclusion principle (Proposition~\ref{prp:IncExc}) to compute the probability that $A_v$ occurs for at least one $v \in I$, i.e., at least one bear sitting on some vertex of $I$ guesses correctly.

Assume that no two bears on adjacent vertices guess correctly their hat colors at once; it turns out that if we plug $-g/h$ into all variables of the non-constant terms of $-P_G$, then we get exactly the fraction of all hat arrangements on which the bears win.
The non-constant terms of $P_G$ correspond (up to sign) to the terms of the formula from the inclusion-exclusion principle.
Because of that, we have to plug $-g/h$ into the polynomial $P_G$.

To avoid confusion with the negative fraction $-g/h$, we define \emph{signed independence polynomial} as $Z_G(\bx) = P_G(-\bx)$, i.e.,
\[Z_G(\bx) = \sum_{\mathclap{\substack{I \subseteq V\\I \text{ independent set}}}} \left(-1\right)^{|I|}\prod_{v \in I} x_v.\]
We also introduce the monovariate signed independence polynomial $U_G(x)$ obtained by plugging $x$ for each variable $x_v$ of $Z_G$.

Note that the constant term of any independence polynomial $P_G(\bx)$ equals to $1$, arising from taking $I = \emptyset$ in the sum from the definition of $P_G$.
When $U_G(g/h) = 0$ and no two adjacent bears guess correctly at the same time, then the bears win the game $(G,h,g)$ because the fraction of all hat arrangements, on which at least one bear guesses correctly, is exactly $1$, however, the proof is far from trivial.

Slightly abusing the notation, we use $Z_{G'}(\bx)$ to denote the independence polynomial of an induced subgraph $G'$ with variables $\bx$ restricted to the vertices of $G'$.
The independence polynomial $P_G$ can be expanded according to a vertex $v \in V$ in the following way.
\[
 P_G(\bx) = P_{G\setminus \{v\}}(\bx) + x_v P_{G \setminus N[v]}(\bx)
\]
The analogous expansions hold for the polynomials $Z_G$ and $U_G$ as well.
This expansion follows from the fact that for any independent set $I$ of $G$, it holds that either $v$ is not in $I$ (the first term of the expansion), or $v$ is in $I$ but in that case, no neighbor of $v$ is in $I$ (the second term).
The formal proof of this expansion of $P_G$ was provided by Hoede and Li~\cite{Hoede94IndPoly}.

For a graph $G$, we let $\cR(G)$ denote the set of all vectors $\br \in [0,\infty)^V$ such that $Z_G(\bw) > 0$ for all $0 \le\bw\le \br$, where the comparison is done entry-wise.
For the monovariate independence polynomial $U_G$, an analogous set to $\cR(G)$ would be exactly the real interval $[0,r)$ where $r$ is the smallest positive root of $U_G$.
(Note that $Z_G({\bf 0})=1$ and $U_G(0)=1$.)

Our first connection of the independence polynomial to the hat guessing game comes in the shape of a sufficient condition for bears to lose.
Consider the following beautiful connection between the Lov{\'a}sz Local Lemma and the independence polynomial due to Scott and Sokal~\cite{scott_sokal_2006}.

\begin{theorem}[\cite{scott_sokal_2006} Theorem 4.1]
\label{thm:LLL}
Let $G =(V,E)$ be a graph and let $(A_v)_{v \in V}$ be a family of events on some probability space such that for every $v$, the event $A_v$ is independent of $\{A_w \mid w \not\in N[v]\}$.
Suppose that $\bp \in [0,1]^V$ is a vector of real numbers such that for each $v$ we have $P(A_v) \le p_v$ and $\bp \in \cR(G)$.
Then
\[P\bigl(\bigcap_{v \in V}\bar{A_v}\bigr) \ge Z_G(\bp) > 0.\]
\end{theorem}

\begin{proposition}
\label{prop:Losing}
A hat guessing game $\cH = (G=(V,E),\bh,\bg)$ is losing whenever $\br \in \cR(G)$ where $\br = (g_v/h_v)_{v \in V}$.
\end{proposition}

\begin{proof}
Suppose for a contradiction that $\cH$ is winning, and fix a strategy of the bears.
We let the demon assign a hat to each bear uniformly at random and independently from the other bears.
Let $A_v$ be the event that the bear on the vertex $v$ guesses correctly.
Observe, that $P(A_v) = \frac{g_v}{h_v}$ and the probability that the bears lose is precisely $P\bigl(\bigcap_{v \in V}\bar{A_v}\bigr)$.

Let us show that the event $A_v$ is independent of all events $A_w$ such that $w \not \in N[v]$.
Observe, that fixing arbitrary hat arrangement $\varphi$ on $V \setminus \{v\}$ uniquely determines the guesses of bears on all vertices except for $N(v)$.
In particular, for every vertex $w \not \in N[v]$, we know whether the bear on $w$ guessed correctly and thus the probability of $A_w$ conditioned by $\varphi$ is either 0 or 1.
On the other hand, the probability of $A_v$ conditioned by $\varphi$ is still $\frac{g_v}{h_v}$.
Therefore, $A_v$ is independent of any subset of $\{A_w \mid w \not\in N[v]\}$.

The claim follows since the graph $G$ and vector $\br$ satisfies the conditions of Theorem \ref{thm:LLL} and we obtain that $P(\bigcap_{v \in V}\bar{A_v}) \ge Z_G(\br) > 0$.
Therefore, there exists some hat arrangement in which all bears guess incorrectly.
\end{proof}

A strategy for a hat guessing game $\cH$ is \emph{perfect} if it is winning and in every hat arrangement, no two bears that guess correctly are on adjacent vertices.
We remark that perfect strategies exist, for example the strategy for a single-guessing game on a clique $K_n$ and exactly $n$ colors~\cite{Kokhas2021CliquesI}, or for a multi-guessing game on a clique $K_n$ and $h / g = n$ (Corollary~\ref{cor:Clique}).
The following proposition shows that a perfect strategy can occur only when $\br = (g_v/h_v)_{v \in V}$ (note $g_v \le h_v$ by definition) lies in some sense just outside of $\cR(G)$.

\begin{proposition}
\label{prop:Perfect}
If there is a perfect strategy for the hat guessing game $(G,\bh,\bg)$ then for $\br = (g_v/h_v)_{v \in V}$ we have that $Z_G(\br) = 0$ and $Z_G(\bw) \ge 0$ for every $0 \le \bw \le \br$.
\end{proposition}

\begin{proof}
Fix a perfect strategy and set $m = \prod_{v \in V} h_v$ to be the total number of possible hat arrangements.
For any subset $S \subseteq V$, let $n_S$ be the number of hat arrangements such that every bear on vertex $v \in S$ guesses correctly (other bears are not forbidden from guessing correctly).
We claim that for any independent set $I \subseteq V$, we have $n_I = m\cdot \prod_{v \in I}\frac{g_v}{h_v}$.

Observe that by assigning the hats to the bears on $V \setminus I$, we fix the guesses of all bears on $I$.
Every bear on a vertex $v \in I$ guesses correctly exactly $g_v$ out of $h_v$ of his hat assignments.
Thus the total number of hat arrangements where every bear on the independent set $I$ guesses correctly is exactly
\[n_I = \prod_{v \in V \setminus I}h_v \cdot \prod_{v \in I}g_v = m\cdot \prod_{v \in I}\frac{g_v}{h_v}. \]

On the other hand, the perfect strategy guarantees that for any non-empty $S$ that is not an independent set, $n_S = 0$.
This allows us to use the inclusion-exclusion principle and count the exact total amount of hat arrangements such that at least one bear guesses correctly
\begin{align*}
\sum_{\emptyset\neq S \subseteq V} (-1)^{|S|+1} n_S &=
\sum_{\mathclap{\substack{\emptyset\neq I \subseteq V\\ I \text{ independent}}}}\ (-1)^{|I|+1} n_I =
m \cdot \sum_{\mathclap{\substack{\emptyset\neq I \subseteq V\\ I \text{ independent}}}}\ (-1)^{|I|+1}\prod_{v \in I}\frac{g_v}{h_v} = \\
&= m \cdot (1 - Z_G(\br)).
\end{align*}

Finally, the total amount of hat arrangements when at least one bear guesses correctly must be exactly $m$ since the bears win.
Therefore, we get $Z_G(\br) = 0$.

We prove the remaining claim in two steps.
First, we show that for every induced subgraph $G'$ of $G$ it holds that $Z_{G'}(\br) \ge 0$.
To that end, consider a modified hat guessing game where only bears on the vertices of $G'$ are allowed to guess and they play according to the original perfect strategy.
By the same argument as before, we can count the total amount of hat arrangements that are guessed correctly by this modified strategy as $m \cdot (1 - Z_{G'}(\br))$.
It implies $Z_{G'}(\br) \ge 0$ as the total number of hat arrangements is $m$.

Now consider any $0 \le\bw \le \br $.
Let $v_1, \ldots, v_n$ be an arbitrary ordering of the vertices of $G$ and let us define vectors $\bw^i$ for $0 \le i \le n$ as
\[
w^i_u = \begin{cases}
w_u \;&\text{if $u = v_j$ for $j \le i$, }\\
r_u &\text{if $u = v_j$ for $j > i$. }
\end{cases}
\]
Notice that $\bw^0 = \br$, $\bw^n = \bw$, and the vectors $\bw^i$ correspond to switching the coordinates of $\br$ into the coordinates of $\bw$ one by one.
We prove by induction on $i$ that for every induced subgraph $G'$ of $G$ it holds that $Z_{G'}(\bw^i) \ge 0$.

We already proved the fact for $i=0$.
Let $i \ge 1$ and let $G'$ be an arbitrary induced subgraph of $G$.
If $G'$ does not contain $v_i$ then $Z_{G'}(\bw^i) = Z_{G'}(\bw^{i-1}) \ge 0$ and we are done.
Otherwise, we have
\begin{align*}
Z_{G'}(\bw^i) &= Z_{G'\setminus \{v_i\}}(\bw^i) - w_{v_i} Z_{G'\setminus N[v_i]}(\bw^i) \\
&\ge Z_{G'\setminus \{v_i\}}(\bw^{i-1}) - r_{v_i} Z_{G'\setminus N[v_i]}(\bw^{i-1}) =
Z_{G'}(\bw^{i-1}) \ge 0
\end{align*}
where we first partition the independent sets of $G'$ according to their incidence with $v_i$ and then replace $\bw^i$ with $\bw^{i-1}$ where the inequality holds since $w_{v_i} \le r_{v_i}$ and $Z_{G'\setminus N(v_i)}(\bw^{i-1}) \ge 0$ from induction.
Finally, we notice that we obtained the independent polynomial $Z_{G'}$ evaluated in $\bw^{i-1}$ and apply induction.
Thus, $Z_G(\bw) \geq 0$ as $\bw = \bw^n$ and $G$ is an induced subgraph of itself.
\end{proof}

Scott and Sokal~\cite[Corollary 2.20]{scott_sokal_2006} proved that $Z_G(\bw) \ge 0$ for every $0 \le \bw \le \br$ if and only if $\br$ lies in the closure of $\cR(G)$.
Therefore, Proposition~\ref{prop:Perfect} further implies that if a perfect strategy for the game $(G,\bh,\bg)$ exists, then $\br = (g_v/h_v)_{v \in V}$ lies in the closure of $\cR(G)$.
And since $\br$ cannot lie inside $\cR(G)$ due to Proposition~\ref{prop:Losing}, it must belong to the boundary of the set $\cR(G)$.

The natural question is what happens outside of the closure of $\cR(G)$.
We proceed to answer this question for chordal graphs.

A graph $G$ is \emph{chordal} if every cycle of length at least $4$ has a chord.
For our purposes, it is more convenient to work with a different equivalent definition of chordal graphs.
For a graph $G = (V,E)$, a \emph{clique tree of $G$} is a tree $T$ whose vertex set is precisely the subsets of $V$ that induce maximal cliques in $G$ and for each $v \in V$ the vertices of $T$ containing $v$ induces a connected subtree.
Gavril~\cite{Gavril74CliqueTree} showed that $G$ is chordal if and only if there exists a clique tree of $G$.

\begin{theorem}
\label{thm:Chordal}
Let $G = (V,E)$ be a chordal graph and let $\br = (r_v)_{v \in V}$ be a vector of rational numbers from the interval $[0,1]$.
If $\br \not\in \cR(G)$ then there are vectors $\bg,\bh \in \N^V$ such that $g_v/h_v \le r_v$ for every $v \in V$ and the hat guessing game $(G,\bh,\bg)$ is winning.
\end{theorem}

\begin{proof}
We prove the theorem by induction on the size of the clique tree of $G$.
Let $0 \le\bw \le \br$ be a witness that $\br \not\in \cR(G)$, i.e., $Z_G(\bw) \le 0$.

If $G$ is itself a complete graph, then $Z_G(\bw) \le 0$ implies that $\sum_{v \in V}w_v \ge 1$ and $\sum_{v \in V}r_v \ge \sum_{v \in V}w_v \ge 1$.
Thus, if we take the minimal vectors $\bg,\bh \in \N^V$ such that $g_v/h_v = r_v$ for each $v$, the assumptions of Theorem~\ref{thm:Clique} are satisfied and the hat guessing game $(G,\bh,\bg)$ is winning.

Otherwise, the clique tree of $G$ contains at least 2 vertices and we pick its arbitrary leaf $C$.
Let $R$ be the set of vertices that belong only to the clique $C$, and let $S = C \setminus R$.
We aim to split the graph into $G' = G[V\setminus R]$ and $G[C]$, apply induction to obtain winning strategies on these graphs, and then combine them into a winning strategy on $G$; see \Cref{fig:ChordalExample}.

If $\sum_{v \in C} r_v \ge 1$, then the game is winning already on the clique $G[C]$ due to Theorem~\ref{thm:Clique}, by letting $g_v/h_v = r_v$ for each $v \in C$.
Therefore, we can assume $\sum_{v \in C} r_v < 1$ which implies $\sum_{v \in C} w_v < 1$.
We define vectors $\bw' = (w'_v)_{v \in  V \setminus R}$ and $\br' =(r'_v)_{v \in V \setminus R}$ as
\begin{align*}
\begin{array}{l}
w'_v = \begin{cases}
w_v / \alpha_w &\text{if $v \in S$,}\\
w_v &\text{otherwise, and}
\end{cases}
\end{array}\
\qquad
r'_v = \begin{cases}
r_v / \alpha_r &\text{if $v \in S$,}\\
r_v &\text{otherwise,}
\end{cases}
\end{align*}
where $\alpha_r = 1 - \sum_{v \in R} r_v$ and $\alpha_w = 1 - \sum_{v \in R} w_v$.
Observe that $0 < \alpha_r \le \alpha_w$ and that for every $v \in V \setminus R$ we have $0 \le w'_v \le r'_v \le 1$.
In other words, $\br'$ and $\bw'$ are both vectors of numbers from $[0,1]$ such that $\bw' \le \br'$.

To simplify the rest of the proof, we introduce the following notation.
For any $u \in V$, let $Z_G(\bx;u)$ denote the independence polynomial restricted only to the independent sets containing $u$, i.e.,
\[Z_G(\bx; u) = \sum_{\mathclap{\substack{u \in I \subseteq V\\ I \text{ independent}}}} \; \left(-1\right)^{|I|} \prod_{v \in I} x_v.\]

With this in hand, we proceed to show that $Z_{G'}(\bw') = Z_G(\bw) / \alpha_w$.

\begin{align}
Z_G(\bw) &= \sum_{v \in R} Z_G(\bw; v) + \sum_{v \in S}Z_G(\bw; v) + Z_{G\setminus C}(\bw) \label{eq1}\\
&= \left(1 - \sum_{v \in R} w_v\right) \cdot Z_{G\setminus C}(\bw) + \sum_{v \in S}Z_{G\setminus R}(\bw;v) \label{eq2}\\
&= \alpha_w \cdot Z_{G\setminus C}(\bw') + \alpha_w \cdot \sum_{v \in S}Z_{G\setminus R}(\bw';v) \label{eq3} \\
&= \alpha_w \cdot  Z_{G\setminus R}(\bw') = \alpha_w \cdot  Z_{G'}(\bw') \label{eq4}
\end{align}

In (\ref{eq1}), we partition the independent sets in $G$ depending on their incidence with $C$.
The line (\ref{eq2}) follows since every independent set intersecting $R$ in $G$ can be written as a union of $v\in R$ and an independent set in $G\setminus C$ which allows us to collect the first and third terms.
At the same time, all independent sets intersecting $S$ in $G$  can be regarded as independent sets intersecting $S$ in $G \setminus R$.
In (\ref{eq3}), we replace $\bw$ with $\bw'$ which scales each term in the second sum by the factor $w_v/w'_v = \alpha_w$.
Finally, notice that the terms in (\ref{eq3}) describe (up to scaling by $\alpha_w$) the independent sets in $G \setminus R$ partitioned according to their incidence with $S$.
We collect them in (\ref{eq4}).

Since $\alpha_w > 0$ and $Z_G(\bw) \le 0$, we have $Z_{G'}(\bw') \le 0$ which witnesses that $\br' \not \in \cR(G')$.
Therefore, we can apply induction on $G'$ and $\br'$ to obtain functions $\bh', \bg'$ such that the hat guessing game $(G', \bh', \bg')$ is winning and $g'_v/h'_v \le r'_v$ for each vertex $v$.

Let $G''$ be the graph obtained from the clique $G[C]$ by contracting $S$ to a single vertex $u$ and define the vector $\br'' = (r''_v)_{v \in R \cup \{u\}}$ as
\[
r''_v =
\begin{cases}
r_v \;&\text{if $v \in R$,}\\
\alpha_r &\text{if $v = u$.}
\end{cases}
\]
Observe that $G$ is precisely the clique join of $G'$ and $G''$ with respect to $S$ and $w$.
Since $r''_u + \sum_{v \in R}{r''_v} = 1$, we can take the minimal vectors $\bh'', \bg'' \in \N^V$ such that $g''_v/h''_v = r_v$ for every $v$ and apply Theorem~\ref{thm:Clique} on $G''$ to show that the hat guessing game $(G'', \bh'', \bg'')$ is winning.
Finally, we construct the desired winning strategy by combining the two graphs and their respective strategies using Lemma~\ref{lem:CliqueJoin} since $r'_v \cdot r''_v = r_v$ for every $v \in S$.
\end{proof}

\begin{figure}[h]
  \centering
  \includegraphics{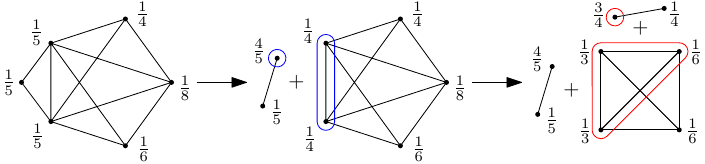}
  \caption{
      Application of Theorem~\ref{thm:Clique} on a chordal graph $G$ with vector $\br \in \cR(G)$.
      In each step, we highlight the clique $S$ and vertex $w$ that are used for Lemma~\ref{lem:CliqueJoin} to inductively build a strategy for $G$ from strategies on cliques given by Theorem~\ref{thm:Clique}.
	  Note that the number of colors and guesses may differ from the depicted ratios by a multiplicative factor.
  }%
  \label{fig:ChordalExample}
\end{figure}

Theorem~\ref{thm:Chordal} applied for the uniform polynomial $U_G$ immediately gives us the following corollary.
\begin{corollary}
\label{cor:Trees}
For any chordal graph $G$, the fractional hat chromatic number $\hat{\mu}(G)$ is equal to $1/r$ where $r$ is the smallest positive root of $U_G(x)$.
\end{corollary}

\begin{proof}
Theorem~\ref{thm:Chordal} implies that $\hat{\mu}(G) \ge 1/r$.
For the other direction, let $(w_i)_{i \in \N}$ be a sequence of rational numbers such that $w_i < r$ for every $i$ and $\lim_{i \to \infty} w_i = r$.
Set $\bw_i = (w_i)_{v\in V}$.
Scott and Sokal~\cite[Thereom~2.10]{scott_sokal_2006} prove that $\br \in \cR(G)$ if and only if there is a path in $[0, \infty)^V$ connecting $\bf 0$ and $\br$ such that $Z_G(\bp) > 0$ for any $\bp$ on the path.
Taking the path $\{\lambda \bw_i \mid \lambda \in [0,1]\}$, we see that $Z_G(\lambda \bw_i) = U_G(\lambda \cdot w_i) > 0$ and thus $\bw_i \in \cR(G)$ for every $i$.
Therefore by Proposition~\ref{prop:Losing}, the hat guessing game $(G, h, g)$ is losing for any $h,g$ such that $g/h = w_i$ and $\hat{\mu}(G) \le 1/w_i$ for every $i$.
It follows that $\hat{\mu}(G) \le 1/r$.
\end{proof}

We would like to remark that the proof of Theorem~\ref{thm:Chordal} (and also Theorem~\ref{thm:Clique}) is constructive in the sense that given a graph $G$ and a vector $\br$ it either greedily finds vectors $\bg, \bh \in \N^V$ such that $g_v/h_v \le r_v$ together with a succinct representation of a winning strategy on $(G,\bh,\bg)$ or it reaches a contradiction if $\br \in \cR(G)$.
Moreover, it is easy to see that it can be implemented to run in polynomial time if the clique tree of $G$ is provided.
Combining it with the well-known fact that a clique tree of a chordal graph can be obtained in polynomial time (see Blair and Peyton~\cite{Blair1993Chordal}) we get the following corollary.

\begin{corollary}
\label{cor:Algo}
There is a polynomial-time algorithm that for a chordal graph $G = (V,E)$ and vector $\br$ decides whether $\br \in \cR(G)$.
Moreover, if $\br \not\in \cR(G)$ it outputs vectors $\bh, \bg \in \N^V$ such that $g_v/h_v \le r_v$ for every $v \in V$, together with a polynomial-size representation of a winning strategy for the hat guessing game $(G,\bh,\bg)$.
\end{corollary}
This result is consistent with the fact that chordal graphs are in general well-behaved with respect to Lov{\'a}sz Local Lemma---Pegden~\cite{Pegden12LLL} showed that for a chordal graph $G$, we can decide in polynomial time whether a given vector $\br$ belongs to $\cR(G)$.
We finish this section by presenting an algorithm that computes the fractional hat chromatic number of chordal graphs.

\begin{theorem}
\label{thm:AlgoValue}
 There is an algorithm $\cA$ such that given a chordal graph $G$ as an input, it approximates $\hat{\mu}(G)$ up to an additive error $1/2^k$.
 The running time of $\cA$ is $2k\cdot \textit{poly}(n)$, where $n$ is the number of vertices of $G$.
 Moreover, if $\hat{\mu}(G)$ is rational, then the algorithm $\cA$ outputs the exact value of $\hat{\mu}(G)$.
\end{theorem}

\begin{proof}
 First, suppose that $\hat{\mu}(G)$ is rational.
 Let $\hat{\mu}(G) = q/p$ for coprimes $p,q \in \N$.
 By Corollary~\ref{cor:Trees}, $1/\hat{\mu}(G) = p/q$ is the smallest positive root of the polynomial $U_G$.
 Let $U_G(x) = a_d x^d + \dots a_1 x + a_0$.
 Note that $a_0 = 1$ and for each $i \leq d$ holds that $|a_i| \leq 2^n$ because $|a_i|$ is exactly the number of independent sets of size $i$ in the graph $G$.
 By the rational root theorem (Theorem~\ref{thm:RationalRoot}), it holds that $p = 1$ and $q \leq 2^n$.

 The algorithm $\cA$ repeats a halving procedure which works as follows.
 We set the initial bounds $\ell_0 = 0$ and $u_0 = 1$.
 In a step $i$, let $r_i = (\ell_i + u_i) / 2$.
 We run the algorithm by Corollary~\ref{cor:Algo} to test if there are $h_i,g_i \in \N$ such that $g_i/h_i \leq r_i$ and the game $\cH_i = (G,h_i,g_i)$ is winning.
 If so, we set new bounds $\ell_{i+1} = \ell_i$ and $u_{i+1} = r_i$.
 On the other hand, if $\cH_i$ is not winning then we set $\ell_{i + 1} = r_i$ and $u_{i+1} = u_i$.
 Thus, for each $i$ it holds that $\ell_i \leq 1/\hat{\mu}(G) \leq u_i$.

 We make $s = \max\{2k, 3n\}$ steps.
 The length of the real interval $I_{s} = [\ell_{s}, u_{s}]$ is at most $1/2^{3n}$.
 It is easy to verify that the interval $I_{s}$ contains at most one rational number $1/q$ for $q \leq 2^n$.
 If so, we output the number $q$.
 Otherwise, we output a number $t$ such that $1/t$ is an arbitrary number in $I_s$.

 If $\hat{\mu}(G)$ is rational, then by Corollary~\ref{cor:Trees} and by the discussion above we found its value.
 Otherwise, we know that $|1/\hat{\mu}(G) - 1/t| \leq 1/2^{s}$ as the length of $I_s$ is exactly $1/2^{s}$.
 Since $s \geq 3n$ and $1/t \geq 1/n$ by Corollary~\ref{cor:Clique}, it follows by easy calculation that $|\hat{\mu}(G) - t| \leq 1/2^{s/2} \leq 1/2^k$.
 Thus, even if $\hat{\mu}(G)$ is irrational, then we estimate it with precision $1/2^k$.

 We ran the halving procedure at most $2k$-times and during each step we run the poly-time algorithm given by Corollary~\ref{cor:Algo}.
 Thus, the running time of $\cA$ is at most $2k\cdot \textit{poly}(n)$.
\end{proof}

\section{Applications}
\label{sec:Applications}
In this section, we present applications of the relation between the hat guessing game and independence polynomials which was presented in the previous section.

\subsection{Fractional Hat Chromatic Number is Almost Linear in the Maximum Degree}
\label{sec:Paths}
First, we prove that $\hat{\mu}(G)$ is asymptotically equal to $\Delta(G)$ up to a logarithmic factor.

\begin{proposition}
\label{prp:MaxDegreeLB}
The fractional hat chromatic number of any graph $G = (V,E)$ is at least $\Omega(\Delta/ \log \Delta)$.
\end{proposition}
\begin{proof}
Let $H$ be a subgraph of $G$.
Note that $\hat{\mu}(H) \leq \hat{\mu}(G)$ as the bears can use a winning strategy for $H$ in $G$.
Let $S$ be a star of $\Delta(G) = \Delta$ leaves.
The graph $G$ contains $S$ as a subgraph.
We prove the proposition by giving a lower bound for $\hat{\mu}(S)$.

By Corollary~\ref{cor:Trees}, we have that $r = 1/\hat{\mu}(S)$ is the smallest positive root of $U_S(x)$.
The independence polynomial of $S$ is
\[
 U_S(x) = - x + \sum_{i = 0}^\Delta \binom{\Delta}{i} (-x)^i  = (1-x)^\Delta - x.
\]
The term $-x$ is given by the independent set containing only the vertex of degree $\Delta$.
The sum is given by all independent sets consisting of leaves of $S$.
Thus, it must hold that $(1 - r)^\Delta = r$.
By simple calculation, we conclude that $r = \Theta\bigl(\log \Delta / \Delta\bigr)$, which implies the assertion of the proposition.
\end{proof}

Farnik~\cite{Farnik2015Thesis} proved that $\mu_g(G) \in O\bigl(g\cdot \Delta(G)\bigr)$, from which we can deduce that $\hat{\mu}(G) \in O\bigl(\Delta(G)\bigr)$.
It gives with Proposition~\ref{prp:MaxDegreeLB} the following corollary that $\hat{\mu}(G)$ is almost linear in $\Delta(G)$.

\begin{corollary}
 \label{cor:DegreeBound}
 For any graph $G$, it holds that $\hat{\mu}(G) \in \Omega(\Delta / \log \Delta)$ and $\hat{\mu}(G) \in O(\Delta)$.
\end{corollary}

\subsection{Paths and Cycles}
In this section, we discuss the precise value of $\hat{\mu}$ of paths and cycles. 
It follows from Corollary~\ref{cor:DegreeBound}, that $\hat{\mu}(P_n)$ and $\hat{\mu}(C_n)$ are upper bounded by constants.
We prove that the fractional hat chromatic number of paths and cycles goes to 4 with their increasing length.

For a proof, we need a version of Lov\'{a}sz local lemma proved by Shearer.
\begin{lemma}[Shearer~\cite{Shearer1985LocalLemma}]
\label{lem:Shearer}
 Let $A_1,\dots,A_n$ be events such that each event is independent on all but at most $d$ other events.
 Let the probability of any events $A_i$ is at most $p$.
 If $d > 1$ and $p < \frac{(d-1)^{d-1}}{d^d}$, then there is non-zero probability that none of the events $A_1,\dots,A_n$ occurs. 
\end{lemma}

\begin{proposition}
\label{prp:Paths}
$
 \lim_{n \to \infty} \hat{\mu}(P_n) = \lim_{n \to \infty} \hat{\mu}(C_n) = 4
$ 
\end{proposition}
\begin{proof}
 First, we prove the lower bound for paths.
 Let $\varepsilon > 0$.
 We construct a sufficiently long path $P = (V, E)$ and vectors $\bh, \bg \in \N^V$ such that a hat guessing game $(P, \bh, \bg)$ is winning and $g_v/{h_v} \leq 1/4 + \varepsilon$.
 Thus, we can conclude that for every $\delta > 0$ there is $n$ such that $\hat{\mu}(P_n) \geq 4 - \delta$, i.e., $\lim_{n \to \infty} \hat{\mu}(P_n) \geq 4$.
 The same lower bound holds for cycles as they contain paths as subgraphs.
 
 We construct the path $P$ iteratively.
 Let $P^0$ be a path consisting of one edge $e_0 = \{v_0,u_0\}$.
 We set $\bg^0_{v_0} = \bg^0_{u_0} = 1$ and $\bh^0_{v_0} = \bh^0_{u_0} = 2$.
 By Theorem~\ref{thm:Clique}, the game $(P^0, \bh^0, \bg^0)$ is winning.
 
 Now, we want to construct a game $\cH_{i+1} = (P^{i+1},\bh^{i+1},\bg^{i+1})$ from $(P^{i},\bh^i,\bg^i)$.
 Let $v_i$ and $u_i$ be the endpoints of $P^i$.
 We will maintain the invariant that $g^i_{v_i} = g^i_{u_i}$ and $h^i_{v_i} = h^i_{u_i}$ and let us denote the ratio $g^i_{v_i}/h^i_{v_i}$ by $r_i$.
 We construct the paths $P^i$ in a way such that $r_i = \frac{1}{2} - i\cdot \varepsilon$.
 Note that this equality holds for the game $(P^0, \bh^0, \bg^0)$.
 
 Let $P'$ be a path consisting of one edge $e' = \{w,w'\}$ and we set $\bg'$ and $\bh'$ in such a way that $g'_w/h'_w = 1/2 + (i+1)\cdot \varepsilon$ and $g'_{w'}/h'_{w'} = {1}/{2} - (i+1)\cdot \varepsilon$.
 Again by Theorem~\ref{thm:Clique}, the game $(P',\bh',\bg')$ is winning.
 To create the path $P^{i+1}$, we join two copies of $P'$ to $P^i$ using Lemma~\ref{lem:CliqueJoin}.
 More formally, we join one copy of $P'$ by identifying $w$ and $u_i$ and the second copy by identifying $w$ and $v_i$.
 Thus, the endpoints $u_{i+1}$ and $v_{i+1}$ of $P^{i+1}$ are copies of $w'$.
 By Lemma~\ref{lem:CliqueJoin}, we get a winning game $\cH_{i + 1} = (P^{i+1}, \bh^{i+1}, \bg^{i+1})$.
 For a sketch of construction of the game $\cH_{i+1}$, see Figure~\ref{fig:Path}.
 Note that indeed $r_{i+1} = \frac{1}{2} - (i+1)\cdot \varepsilon$.

 \begin{figure}
  \centering
  \includegraphics{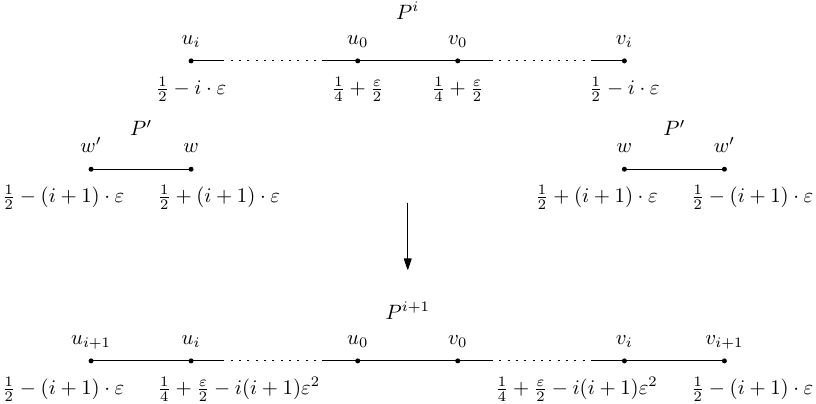}
  \caption{A sketch of construction of the game $\cH_{i+1}$. The formulas below vertices are the fractions $g_v/h_v$.}
  \label{fig:Path}
 \end{figure}
 
 We end this process after $k = \left\lceil\frac{1}{4\varepsilon}\right \rceil$ steps.
 Thus, it holds that $r_{k} = \frac{1}{2} - k\cdot \varepsilon \leq \frac{1}{4}$.
 On the other hand, it holds for each $0 \leq i < k$ by Lemma~\ref{lem:CliqueJoin} that 
 \[
  \frac{g^k_{v_i}}{h^k_{v_i}} = \frac{g^k_{u_i}}{h^k_{u_i}} = \left(\frac{1}{2} - i\cdot \varepsilon\right)\cdot \left(\frac{1}{2} + (i+1)\cdot \varepsilon\right) = \frac{1}{4} + \frac{\varepsilon}{2} - i(i + 1)\varepsilon^2.
 \]
 Thus, for each vertex $v$ of $P^k$ holds that $\frac{g^k_v}{h^k_v} \leq \frac{1}{4} + \varepsilon$ as claimed.

 Now, we prove the upper bound.
 Let $\cH = (G,h,g)$ be a game such that $G$ is a path or a cycle and $\frac{h}{g} > 4$.
 We will prove that bears lose $\cH$, which implies that $\lim_{n \to \infty} \hat{\mu}(P_n),\lim_{n \to \infty} \hat{\mu}(C_n) \leq 4$.
 Let us fix some strategy of bears and the demon gives each bear a hat of random color (chosen uniformly and independently).
 We denote $A_v$ an event that the bear on $v$ guesses correctly.
 Then, $\Pr[A_v] = \frac{g}{h} < \frac{1}{4}$.
 Since the maximum degree in $G$ is 2, each event $A_v$ might depend only on at most 2 other events.
 By Lemma~\ref{lem:Shearer}, for events $(A_v)_{v \in V(G)}$ and $d = 2$, we have that no event $A_v$ occurs with non-zero probability.
 Thus, there is a hat arrangement such that no bear guesses correctly.
\end{proof}

We remark that Proposition~\ref{prp:Paths} follows also from the results of Scott and Sokal~\cite{scott_sokal_2006} as they proved that the small positive roots of $U_{P_n}$ and $U_{C_n}$ go to $1/4$ when $n$ goes to infinity.
However, their proof is purely algebraic whereas we provide a combinatorial proof.

Further, we discuss the value of $\hat{\mu} = \hat{\mu}(P_3)$.
By Corollary~\ref{cor:Trees}, we have that $1/\hat{\mu}$ is the smallest positive root of $U_{P_3}(x) = x^2 - 3x + 1$.
Thus, $1/\hat{\mu} = (3 - \sqrt{5})/2$. 
By Theorem~\ref{thm:Chordal}, it holds that for any $p,q \in \N$ such that $\hat{\mu} \leq p/q$ there are $g,h \in \N$ such that $p/q = h/g$ and the game $(P_3,h,g)$ is winning.
However, the strategy from the proof gives us $h = p\cdot (p - q)$ and $g = q\cdot (p-q)$.
We present a sequence $\left(h_i/g_i\right)_{i \in \N}$ such that the sequence goes to $\hat{\mu}$, for each $i$ the numbers $h_i$ and $g_i$ are coprime, and the game $(P_3,h_i,g_i)$ is winning for each $i$.
Thus, we present a strategy that is in some sense more efficient than the strategy given by the proof of Theorem~\ref{thm:Chordal} as the general strategy for $P_3$ does not produce numbers $g$ and $h$ which are coprimes.

First, we present the strategy for $P_3$.
Note that if $1 \geq g/h \geq 1/\hat{\mu}$ (for $g,h \in \N$) then $U_{P_3}\left(g/h\right) = \left(g/h\right)^2 - 3g/h + 1 < 0$.
We change the inequality to $g^2 - 3gh + h^2 < 0$ and prove that for each $g$ and $h$, which satisfy the previous inequality, there is a winning strategy for $(P_3,h,g)$.

\begin{lemma}
 \label{lem:StrategyP3}
 Let $g,h \in \N$ such that $g^2 - 3gh + h^2 < 0$.
 Then, the bears win the game $(P_3,h,g)$.
\end{lemma}
\begin{proof}
 Let $V(P_3) = \{u,v,w\}$ where $v$ and $w$ are the endpoints of the path $P_3$.
 We identify the colors with a set $C = \{0,\dots,h-1\}$.
 Let the bear on $v$ get a hat of color $c_v$.
 The bear on $u$ makes guesses $A_u = \bigl\{c_v, c_v -1,\dots, c_v - (g - 1)\bigr\}$.
 The bear on $w$ makes guesses $A_w = \bigl\{c_v, c_v - \left\lfloor h/g \right\rceil, \dots, c_v - \left\lfloor (g - 1)\cdot h/g \right\rceil\bigr\}$, where $\lfloor x \rceil$ is the nearest integer to $x$ (i.e., standard rounding).
 We compute the guessed colors modulo $h$.

 The bear on $v$ computes two sets of colors $I_u$ and $I_w$ based on the hat colors of bears on $u$ and $w$ such that he will not guess the colors from $I_u \cup I_w$ because if $c_v \in I_u \cup I_w$ then the bear on $u$ or the bear on $w$ would guess correctly (or maybe both of them).
 The guesses of the bear on $u$ is an interval in the set $C$.
 However, the guesses of the bear on $w$ are spread through $C$ as evenly as possible.
 Thus, the intersection $I_u \cap I_w$ is small and $I_u \cup I_w$ is large.

 More formally, let $c_u$ and $c_w$ be hat colors of the bears on $u$ and $w$, respectively.
 Then, $I_u = \{c_u,c_u + 1,\dots,c_u + (g - 1)\}$, and $I_w = \{c_w, c_w + \lfloor h/g \rceil,\dots, c_w + (g - 1)\cdot \lfloor h/g \rceil\}$.
 Again, we compute the elements in the sets modulo $h$.
 Note that if $c_v \in I_u$ then the bear on $u$ guesses correctly because in that case $c_v = c_u + t \pmod{h}$ for some $t < g$ and thus $c_u \in A_u$.
 An analogous property holds for $c_w$.
 Thus, the bear on $v$ does not have to guess the colors from $I_u \cup I_w$.

 We will prove that $\bigl|C \setminus (I_u \cup I_w)\bigr| \leq g$.
 Thus, the bear on $v$ can guess all colors outside $I_u$ and $I_w$ and makes at most $g$ guesses.
 First, we prove that $|I_u \cap I_w| \leq 3g - h$.
 Suppose opposite, that $|I_u \cap I_w| > 3g - h$.
 In such a case, there must be $k$ such that both colors $c_w + \lfloor k \cdot h/g \rceil$ and  $c_w + \lfloor (k + 3g - h) \cdot h/g \rceil$ belong to $I_u$.
 This implies that $\lfloor (k + 3g - h) \cdot h/g \rceil - \lfloor k \cdot h/g \rceil \leq g-1$.
 Applying bounds on the rounded terms, we obtain
\begin{align*}
g-1 &\geq \left\lfloor (k + 3g - h) \cdot \frac{h}{g} \right\rceil - \left\lfloor k \cdot
\frac{h}{g} \right\rceil \\
    &\geq  (k + 3g - h) \cdot \frac{h}{g} - 0.5 - k \cdot \frac{h}{g} - 0.5 \\
    &= (3g-h)\cdot \frac{h}{g} - 1.
\end{align*}

The final inequality implies $g^2 - 3gh + h^2 \geq 0$ which contradicts the assumption of the lemma.
Therefore, the size of the intersection $I_u \cap I_w$ is at most $3g-h$.
It follows that the size of the union $I_u \cup I_w$ is at least $2g - (3g - h) = h - g$ and $\bigl|C \setminus (I_u \cup I_w)\bigr| \leq g$.
\end{proof}

Let $F_i$ be the $i$-th Fibonacci number\footnote{$F_0 = F_1 = 1$ and $F_{i+1} = F_{i - 1} + F_i$.}.
We define $h_i = F_{2i}$ and $g_i = F_{2i - 2}$.
Now, we prove the sequence $(g_i/h_i)_{i \in \N}$ has the sought properties.

\begin{lemma}
 For each $i \in \N$ it holds that $\frac{h_i}{g_i} \leq \hat{\mu}$. Moreover,
 \[
  \lim_{i \to \infty} \frac{h_i}{g_i} = \hat{\mu}.
 \]
\end{lemma}
\begin{proof}
 Note that $1/\hat{\mu} = 1 - \frac{\sqrt{5} - 1}{2} = 1 - \frac{1}{\varphi}$, where $\varphi$ is the golden ratio, i.e., $\varphi = \frac{1 + \sqrt{5}}{2}$.
 It is well-known that fractions $\frac{F_i}{F_{i-1}}$ go to $\varphi$.
 Moreover, $\frac{F_{2i}}{F_{2i-1}} \geq \varphi$.
 Thus, for each $i \in \N$ it holds that
 \[
  \frac{1}{\hat{\mu}} = 1 - \frac{1}{\varphi} \leq 1 - \frac{F_{2i-1}}{F_{2i}} = \frac{F_{2i-2}}{F_{2i}} = \frac{g_i}{h_i}.
 \]
 and the fractions $\frac{h_i}{g_i}$ indeed go to $\hat{\mu}$. 
\end{proof}

\begin{observation}
    Due to Cassini's identity, for each $i \in \N$ holds that
    \begin{equation}\label{eq:Fibonacci}
        g_i^2 - 3g_i h_i + h_i^2 = (h_i-g_i)^2 - h_ig_i = F_{2i-1}^2-F_{2i}F_{2i-2}=(-1)^{2i-1} = -1
    \end{equation}
\end{observation}

\begin{observation}
 For each $i \in \N$ the numbers $h_i$ and $g_i$ are coprime.
\end{observation}
\begin{proof}
 By definition, $g_i = F_{2i-2}$ and $h_i = F_{2i}$.
 \[
  \text{GCD}(F_{2i-2}, F_{2i}) = \text{GCD}(F_{2i-2}, F_{2i-1} + F_{2i-2}) = \text{GCD}(F_{2i-2}, F_{2i-1})
 \]
 It is easy to prove by induction that for each $i \in \N$ it holds that $\text{GCD}(F_{i-1}, F_{i}) = 1$.
\end{proof}

\acknowledgements{We would like to thanks to Milo\v{s} Chrom\'{y}, Micha\l{} D\k{e}bski, Sophie Rehberg, and Pavel Valtr for fruitful discussions at early stage of this project during workshop KAMAK 2019.}

\bibliography{main}

\appendix

\section{The Second Proof of the Non-algorithmic Part of Theorem~\ref{thm:Clique}}
\label{sec:Clique2}

\begin{theorem}[Non-algorithmic part of Theorem~\ref{thm:Clique}]
 Bears win a game $(K_n, \bh, \bg)$ if and only if 
 \[
  \sum_{v \in V(K_n)} \frac{g_v}{h_v} \geq 1.
 \]
\end{theorem}
\begin{proof}[(The second proof of non-algorithmic part of Theorem~\ref{thm:Clique})]
 The proof again follows the proof of Kokhas et al.~\cite{Kokhas2021CliquesI} for the single-guessing game.
 We prove only the ``if'' part.
 Thus, suppose that $\sum_{v \in V(K_n)} \frac{g_v}{h_v} \geq 1$.
 Let $V(K_n) = \{v_1,\dots,v_n\}$.
 We create an auxiliary bipartite graph $G = (V_\ell \dot\cup V_r, E)$.
 In the left partite $V_\ell$ there is a vertex for each possible coloring of hats.
 Thus we can identify each vertex in $V_\ell$ with an $n$-tuple $(c_1,\dots,c_n)$ where $c_i \in [h_{v_i}]$ is some color of the $i$-th bear's hat.
 The set $V_r$ is split into $n$ sets, $V_r = V^1_r \dot\cup \dots \dot\cup V^n_r$.
 For each $v_i \in V(K_n)$ and a tuple $(c_1,\dots,c_{i-1},*,c_{i+1},\dots,c_n)$ we have a $g_{v_i}$ vertices in the set $V^i_r$.
 Thus, the vertices in $V^i_r$ represent what could see the $i$-th bears.
 Each vertex in $V^i_r$ labeled with $(c_1,\dots,c_{i-1},*,c_{i+1},\dots,c_n)$ is connected with vertices in $V_\ell$ labeled with $(c_1,\dots,c_{i-1},c_i,c_{i+1},\dots,c_n)$ for all $c_i \in [h_{v_i}]$.
 Thus, each vertex in $V^i_r$ has degree $h(v_i)$ and each vertex in $V_\ell$ has a degree $\prod_{v \in V(K_n)} g(v)$.
 
 Note that bears win the game if and only if there is a matching in $G$ which covers $V_\ell$.
 Suppose there is such matching $M$.
 Let a bear sitting on a vertex $v_i$ sees colors $c_1,\dots,c_{i - 1},c_{i + 1},\dots,c_n$ and $U \subseteq V_r$ be a set of vertices in $V_r$ labeled by $(c_1,\dots,c_{i - 1},*,c_{i + 1},\dots,c_n)$.
 By construction of $G$, it holds that $|U| = g_{v_i}$.
 Let $N(U)$ be a set of neigbors of $U$ given by the matching $M$, thus, $|N(U)| \leq g(v_i)$.
 Each vertex $u \in N(U)$ has label $(c_1,\dots,c_{i - 1},c^u_i,c_{i + 1},\dots,c_n)$.
 Thus, the bear sitting on $v_i$ guesses colors $c^u_i$ for all $u \in N(U)$.
 It is clear that for each $v \in V(K_n)$, the bear sitting on $v$ guesses at most $g_v$ colors.
 Moreover, since the matching $M$ covers $V_\ell$ at least one bear guesses the color of his hat correctly.
 On the other hand, each winning strategy gives us a matching covering $V_\ell$.
 
 We use Hall's theorem~\cite[Chapter 2]{DiestelBook} to prove there is a matching $M$ covering $V_\ell$ if and only if $\sum_{v \in V(K_n)} \frac{g_v}{h_v} \geq 1$. 
 Let $S \subseteq V_\ell$ be a set of $m$ left vertices.
 Each vertex in $V^i_r$ has at most $h_{v_i}$ neigbors in $S$.
 Since each vertex in $V_\ell$ has $g_{v_i}$ neigbors in $V^i_r$, the set $S$ has at least $g_{v_i} \cdot \frac{m}{h_{v_i}}$ vertices in $V^i_r$.
 Therefore, in total the set $S$ has at least
 \[
  \sum_{v \in V(K_n)} g_v\cdot \frac{m}{h_v} \geq m
 \]
 neighbors in $V_r$.
 We conclude by Hall's theorem, that there is a matching in $G$ which covers $V_\ell$.
\end{proof}

Albeit Hall's theorem is constructive, the size of the auxiliary graph $G$ constructed in the proof could be exponential in $n$.
Thus, this proof can not be used for designing a polynomial algorithm.

\end{document}